\documentclass[letterpaper, 10pt, conference]{ieeeconf}
\IEEEoverridecommandlockouts
\usepackage{graphicx, enumerate}
\usepackage{cite,array, bm}
\usepackage{dsfont, epstopdf}
\usepackage{mathrsfs, subfigure, color}
\usepackage{amssymb, multirow}
\usepackage[cmex10]{amsmath}
\usepackage{subeqnarray, tabularx}
\usepackage{cases, stmaryrd}
\usepackage{algpseudocode}
\usepackage[ruled,vlined,linesnumbered]{algorithm2e}
\usepackage[utf8]{inputenc}
\usepackage{threeparttable}
\usepackage{stackengine}
\usepackage{verbatim}
\usepackage{mathtools}
\usepackage{scalerel}

\DeclareFontFamily{U}{FdSymbolC}{}
\DeclareFontShape{U}{FdSymbolC}{m}{n}{<-> s * FdSymbolC-Book}{}
\DeclareSymbolFont{fdarrows}{U}{FdSymbolC}{m}{n}
\DeclareMathSymbol{\vDdash}{\mathrel}{fdarrows}{254}

\DeclareFontFamily{U}{FdSymbolD}{}
\DeclareFontShape{U}{FdSymbolD}{m}{n}{<-> s * FdSymbolD-Book}{}
\DeclareSymbolFont{fdnarrows}{U}{FdSymbolD}{m}{n}
\DeclareMathSymbol{\nvDdash}{\mathrel}{fdnarrows}{254}
\makeatother

\newtheorem{assumption}{Assumption}
\newtheorem{theorem}{Theorem}

\newtheorem{proposition}{Proposition}
\newtheorem{problem}{Problem}

\newcommand{\opt}{\textrm{opt}}
\newcommand{\con}{\textrm{con}}
\newcommand{\diag}{\mathrm{diag}}
\newcommand{\dett}{\mathrm{det}}

\definecolor{bluencs}{rgb}{0.0, 0.53, 0.74}

\begin{document}

\title{\LARGE \bf Optimal Resource Scheduling and Allocation \\ under Allowable Over-Scheduling
\thanks{This work was supported by the H2020 ERC Consolidator Grant L2C (Grant 864017), the CHIST-ERA 2018 project DRUID-NET, the Walloon Region and the Innoviris Foundation.}
}

\author{Wei Ren, Eleftherios Vlahakis, Nikolaos Athanasopoulos and Rapha\"el M. Jungers
\thanks{W. Ren and R. M. Jungers are with ICTEAM institute, UCLouvain, 1348 Louvain-la-Neuve, Belgium. E. Vlahakis and N. Athanasopoulos are with School of Electronics, Electrical Engineering and Computer Science, Queen’s University Belfast, Northern Ireland, UK. Email: \texttt{\small w.ren@uclouvain.be, e.vlahakis@qub.ac.uk, n.athanasopoulos@qub.ac.uk, raphael.jungers@uclouvain.be}.}
\thanks{W. Ren and E. Vlahakis contributed equally.}
}

\maketitle

\begin{abstract}
This paper studies optimal scheduling and resource allocation under allowable over-scheduling.  Formulating an optimisation problem where over-scheduling is embedded, we derive an optimal solution that can be implemented by means of a new additive increase multiplicative decrease (AIMD) algorithm. After describing the AIMD-like scheduling mechanism as a switching system, we show convergence of the scheme, based on the joint spectral radius of symmetric matrices, and propose two methods for fitting an optimal AIMD tuning to the optimal solution derived. Finally, we demonstrate the overall optimal design strategy via an illustrative example.
\end{abstract}

%%%%%%%%%%%%%%%%%%%%%%%%%%%%%%%%%%%%%%%%%%%%%%%%%%%%%%%%%%%%%%%%%%%%%%%%%%%%%%%%%%%%%%%%%%%%%%%%%%%%%%%%%
\section{Introduction}
\label{sec-intro}
%%%%%%%%%%%%%%%%%%%%%%%%%%%%%%%%%%%%%%%%%%%%%%%%%%%%%%%%%%%%%%%%%%%%%%%%%%%%%%%%%%%%%%%%%%%%%%%%%%%%%%%%%%

With the advances in networking and hardware technology, distributed computing is increasingly adopted in many applications like multi-agent systems \cite{Woolridge2001introduction} and network systems \cite{Kshemkalyani2011distributed}. Distributed computing systems have been proposed in different fields \cite{Hussain2013survey}. In particular, scalability, reliability, information sharing/exchange from remote sources are main motivations for the users of distributed systems \cite{Amir2000opportunity}. For distributed computing systems, many topics have been investigated \cite{Kimura2020packet, Pentelas2020network, Herrera2016resource}, such as power management, request scheduling, resource allocation and system reliability, thereby resulting in different models, algorithms and software tools.

Among these extensively-studied topics, resource allocation is one of most important problems due to its fundamental role in high performance of computing systems \cite{Hussain2013survey}. Resource allocation mechanisms include both scheduling of requests among different nodes and allocation of limited resources that are provisioned to requests entering computing systems. To solve resource allocation problems, many resource allocation models and strategies have been proposed \cite{Hussain2013survey}, such as proportional-share scheduling, market-based and auction-based strategies. To avoid an excess capacity of resources, the over-scheduling phenomena are excluded in most existing strategies \cite{Pentelas2020network, Ren2021optimal, Luna2017mixed}. Over-scheduling phenomena may result in overuse of resources \cite{Pantazoglou2015decentralized} and excess of predefined thresholds \cite{Pentelas2020network, Ren2021optimal, Luna2017mixed}, leading to performance degradation, device damages, economic loss and even environmental threats. Nevertheless, in some scenarios, over-scheduling can be neither avoided in some scenarios nor is always harmful to computing systems \cite{Pantazoglou2015decentralized}. For instance, over-scheduling is required in peak hours and may be caused by allocation strategies. Due to over-scheduling, the boundedness of the waiting time for each request is ensured and no additional techniques \cite{Ren2021optimal} are needed. To the best of our knowledge, there are few works on resource allocation under the over-scheduling.

In this paper, motivated by the potential benefits of over-scheduling, we deal with a resource allocation problem where a bounded over-scheduling is permitted. This effectively allows us to implement an optimal scheduling strategy via a new variation of the additive increase multiplicative decrease (AIMD) algorithm which is known for its robustness properties and scalability. Here, we introduce an average-based AIMD (A-AIMD) algorithm which tackles a scheduling task from an average perspective and ensures bounded waiting time for all requests. Our proposed AIMD-based scheduling mechanism is motivated by recently introduced AIMD-like algorithms utilized for scheduling \cite{Vlahakis2021AIMD, Ren2021optimal}. First, the over-scheduling is embedded into an optimization problem, whose objective is to minimize the sum of the response time and cost of all computing nodes. The optimal strategy provides the bounds on request scheduling and resource allocation without involving the A-AIMD mechanism. Hence, we highlight the inherent over-scheduling feature of the proposed A-AIMD algorithm that essentially guarantees that all requests are scheduled in finite time.

Although convergence and stability of typical AIMD schemes are well-studied in literature, see, e.g., \cite{Vlahakis2021AIMD}, these properties need to be revisited here, primarily, due to a different triggering condition which is, to the best of our knowledge, novel in the context of event-triggered distributed systems. Due to its definition, the proposed A-AIMD scheduling mechanism is intuitively transformed to a switching system the stability of which is shown by boundedness of the joint spectral radius of symmetric matrices \cite{Jungers2009joint}. Having established the convergence property of the scheduling scheme, we then highlight the relation between the magnitude of over-scheduling and the tuning parameters of the proposed A-AIMD setting. We finally propose two (complementary) methods for  parameter design that can straightforwardly be used for optimal AIMD tuning.

The remainder of this paper is structured as follows. Section \ref{sec-notepre} formulates the considered problem. The optimal strategy is proposed in Section \ref{sec-optimalschedul}. The A-AIMD mechanism is addressed in Section \ref{sec-AIMDschedule}. A numerical example is given in Section \ref{sec-example}. Conclusion and future work are presented in Section \ref{sec-conclusion}.

\textbf{Notation.} $\mathbb{N}:=\{0, 1, \ldots\}; \mathbb{N}^{+}:=\{1, 2, \ldots\}; \mathbb{R}:=(-\infty, +\infty)$. $\mathbb{R}^{n}$ is the $n$-dimensional Euclidean space. Given a function $f: \mathbb{R}^{n}\rightarrow\mathbb{R}$ and $x=(x_{1}, \ldots, x_{n})\in\mathbb{R}^{n}$, we denote by $\nabla f(x)$ the gradient of $f$ with respect to $x$. Given a bounded set of matrices $\Sigma=\{A_{1}, \ldots, A_{m}\}$, the joint spectral radius of $\Sigma$ is defined as $\mathsf{JSR}(\Sigma)=\lim_{k\rightarrow\infty}\max_{\sigma\in\{1, \ldots, m\}^{k}}\|A_{\sigma_{k}} \ldots A_{\sigma_{2}}A_{\sigma_{1}}\|^{1/k}$.

%%%%%%%%%%%%%%%%%%%%%%%%%%%%%%%%%%%%%%%%%%%%%%%%%%%%%%%%%%%%%%%%%%%%%%%%%%%%%%%%%%%%%%%%%%%%%%%%%%%%%%%%%%
\section{Problem Formulation}
\label{sec-notepre}
%%%%%%%%%%%%%%%%%%%%%%%%%%%%%%%%%%%%%%%%%%%%%%%%%%%%%%%%%%%%%%%%%%%%%%%%%%%%%%%%%%%%%%%%%%%%%%%%%%%%%%%%%%

We consider  requests arriving at the dispatcher of a computing system illustrated in Fig. \ref{fig-1}. We model incoming requests as a continuous constant process with an arrival rate $\lambda>0$; see also \cite{Vlahakis2021AIMD}. On their arrival, incoming requests are first queued in the dispatcher and then are distributed to $n$ computing nodes. Each computing node scales its service capacity to minimize the processing time and cost.

To describe the number of all queued requests in the dispatcher at each time instant, the variable $\delta(t)\geq0$ is introduced to denote the backlog of the dispatcher. For each computing node ($i\in\mathcal{N}:=\{1, \ldots, n\}$), $u_{i}(t)\geq0$ is the scheduling rate at which requests are distributed to the $i$-th computing node from the dispatcher at time $t$; $\gamma_{i}(t)>0$ is the service rate of the $i$-th computing node at time $t>0$. If $u_{i}(t)>\gamma_{i}(t)$, then the $i$-th computing node may store some requests before their execution. In this respect, similar to the dispatcher, the variable $w_{i}(t)\geq0$ is introduced to denote the backlog of the $i$-th computing node. Since not all computing nodes need to be activated, $u_{i}(t)$ and $\gamma_{i}(t)$ can be zero. The next assumption is made.

\begin{assumption}
\label{asp-1}
For the considered computing system, the following holds.
\begin{enumerate}
  \item For each $i\in\mathcal{N}$, there exists $\gamma^{\max}_{i}>0$ such that $u_{i}(t)<\gamma_{i}(t)\leq\gamma^{\max}_{i}$ for all $t>0$.
  \item No constraints are imposed on the maximum backlog in the dispatcher.
\end{enumerate}
\end{assumption}

In the first item of Assumption \ref{asp-1}, the first part of the inequality $\gamma_{i}(t)\leq\gamma^{\max}_{i}$ comes from the limited capacity of computing nodes, while $u_{i}(t)<\gamma_{i}(t)$ can be imposed by the scheduling and allocation strategy to be designed. The second assumption implies that the dispatcher has sufficiently large capacity to store the incoming requests, which can be satisfied in practice; see, e.g., \cite{Amazon2015unlimited}.

\subsection{Average-based AIMD Mechanism}
\label{subsec-AIMDstrategy}

The classic AIMD mechanism \cite[Ch. 14]{Corless2016aimd} consists of  an additive increase (AI) phase and a multiplicative decrease (MD) phase. In the AI phase, the scheduling rates increase linearly with an additive rate $\alpha_{i}>0$ until the capacity event $\sum_{i=1}^{n}u_{i}=\lambda$ occurs. Immediately after the event, the MD phase is activated and the scheduling rate experiences an instantaneous decrease (i.e., a jump) with a multiplicative factor $\beta_{i}\in(0, 1)$. In this way, the scheduling rates evolve as follows:
\begin{align}
\label{eqn-1}
u_{i}(t)=\beta_{i}u_{i}(t_{k})+\alpha_{i}(t-t_{k}), \quad \forall t\in(t_{k}, t_{k+1}],
\end{align}
where $t_{k}$ is the activation time of the MD phase and $k\in\mathbb{N}$. That is, $\lambda=\sum_{i=1}^{n}u_{i}(t_{k})$ is the event-triggering condition for a typical AIMD scheme \cite{Corless2016aimd}.

\begin{figure}[!t]
\begin{center}
\begin{picture}(40, 85)
\put(-45, -10){\resizebox{40mm}{30mm}{\includegraphics[width=2.5in]{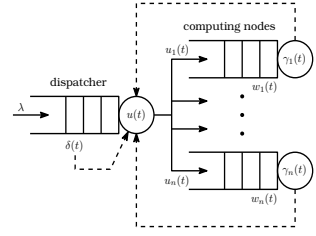}}}
\end{picture}
\end{center}
\caption{Illustration of the request scheduling and resource allocation.}
\label{fig-1}
\end{figure}

For the request scheduling and resource allocation as in Fig. \ref{fig-1}, the dynamics of $\delta(t)$ and $w_{i}(t)$ is derived below \cite{Presti1996bounds}:
\begin{align}
\label{eqn-2}
\begin{aligned}
\dot{\delta}(t)=\lambda-\mathbf{u}(t), \quad \dot{w}_{i}(t)=u_{i}(t)-\gamma_{i}(t),
\end{aligned}
\end{align}
where $\mathbf{u}(t):=\sum^{n}_{i=1}u_{i}(t)$. From \eqref{eqn-2} and Assumption \ref{asp-1}, it holds that $\dot{w}_{i}\leq0$, which implies the boundedness of the backlogs of all computing nodes. However, $\sum_{i=1}^{n}u_{i}=\lambda$ cannot ensure the boundedness of $\delta(t)$, since the over-scheduling is not allowed in the classic AIMD mechanism. To ensure that $\delta(t)$ is finite and bounded for all $t>0$, we propose an AIMD-like algorithm with a novel triggering condition. We call this average-based AIMD (A-AIMD).

Instead of $\sum_{i=1}^{n}u_{i}=\lambda$, the jumps in the A-AIMD mechanism occur when $\delta=0$, that is, no backlog in the dispatcher and thus no requests to be scheduled. Hence, the scheduling rates jump when $\delta=0$, and the evolution of the scheduling rate is given in \eqref{eqn-1}, where $t_{k}$ is such that $\delta(t_{k})=0$. From \eqref{eqn-2}, $\delta(t_{k})=0$ implies
\begin{align}
\label{eqn-3}
\lambda\mathsf{T}_{k}=\scaleobj{0.8}{\int_{t_{k}}^{t_{k+1}}}\scaleobj{0.8}{\sum_{i=1}^{n}}u_{i}(t)dt=\scaleobj{0.8}{\sum_{i=1}^{n}}(\beta_{i} u_{i}(t_{k})+0.5\alpha_{i}\mathsf{T}_{k})\mathsf{T}_{k},
\end{align}
where $\mathsf{T}_{k}:=t_{k+1}-t_{k}$, $k\in\mathbb{N}$. In this case, the event-triggered mechanism (ETM) for the MD phase is
\begin{align}
\label{eqn-4}
t_{k+1}&=\min\scaleobj{1.5}{\{}t>t_{k}: \lambda=\scaleobj{0.8}{\sum_{i=1}^{n}}(\beta_{i}u_{i}(t_{k})+0.5\alpha_{i}\mathsf{T}_{k})\scaleobj{1.5}{\}},
\end{align}
which is the same as the one in \cite{Vlahakis2021AIMD}. We call it A-AIMD triggering condition as  $\beta_{i}u_{i}(t_{k})+0.5\alpha_{i}\mathsf{T}_{k}$ is the average share of each computing node between two successive jumps \cite[Ch. 2]{Corless2016aimd}.

Comparing with the classic AIMD mechanism, the A-AIMD mechanism allows for the over-scheduling. To be specific, if $\lambda=\sum_{i=1}^{n}(\beta_{i}u_{i}(t_{k})+0.5\alpha_{i}\mathsf{T}_{k})$ is satisfied such that the MD phase is activated, then $\sum_{i=1}^{n}(\beta_{i}u_{i}(t_{k})+\alpha_{i}\mathsf{T}_{k})>\lambda$, which implies that over-scheduling is inherent in the A-AIMD mechanism. However, over-scheduling may impose potential threats \cite{Pentelas2020network}, and hence, its permissible magnitude should be subject to system constraints \cite{Pantazoglou2015decentralized}. To avoid the unallowable over-scheduling, we assume
\begin{align}
\label{eqn-5}
\scaleobj{0.8}{\sum_{i=1}^{n}}u_{i}(t_{k})=\scaleobj{0.8}{\sum_{i=1}^{n}}(\beta_{i}u_{i}(t_{k})+\alpha_{i}\mathsf{T}_{k})\leq\phi\lambda,
\end{align}
where $\phi>1$ is given \emph{a priori}. By combining \eqref{eqn-4} and \eqref{eqn-5}, we write the proposed A-AIMD system as follows:
\begin{subequations}
\label{eqn-6}
\begin{align}
\label{eqn-6a}
u_{i}(t)&=\beta_{i}u_{i}(t_{k})+\alpha_{i}(t-t_{k}), \quad \forall t\in(t_{k}, t_{k+1}], \\
\label{eqn-6b}
t_{k+1}&=\min\scaleobj{1.5}{\{}t>t_{k}: \scaleobj{0.8}{\sum_{i=1}^{n}}(\beta_{i}u_{i}(t_{k})+0.5\alpha_{i}\mathsf{T}_{k})=\lambda \nonumber  \\
&\quad  \text{ or } \scaleobj{0.8}{\sum_{i=1}^{n}}(\beta_{i}u_{i}(t_{k})+\alpha_{i}\mathsf{T}_{k})=\phi\lambda\scaleobj{1.5}{\}}.
\end{align}
\end{subequations}
In view of \eqref{eqn-6}, the over-scheduling is clearly confined by the second item of the ETM which can be controlled by parameter $\phi$. One of the main contributions of this paper is to show that A-AIMD system \eqref{eqn-6a} is stable under triggering condition \eqref{eqn-6b}. Hence, the problem to be studied in this paper is summarized below.

\begin{problem}
\label{prob-1}
Consider the distributed computing system in Fig. \ref{fig-1} and the A-AIMD mechanism \eqref{eqn-6}. Given a well-defined objective function,  Assumption \ref{asp-1}, and allowable over-scheduling, (1) determine the optimal strategy for both request scheduling and resource allocation; (2) establish the convergence property of the A-AIMD mechanism; (3) propose methods for AIMD tuning so that the derived optimal strategy is obtained by means of the A-AIMD mechanism.
\end{problem}

%%%%%%%%%%%%%%%%%%%%%%%%%%%%%%%%%%%%%%%%%%%%%%%%%%%%%%%%%%%%%%%%%%%%%%%%%%%%%%%%%%%%%%%%%%%%%%%%%%%%%%%%%%
\section{Optimal Scheduling and Allocation}
\label{sec-optimalschedul}
%%%%%%%%%%%%%%%%%%%%%%%%%%%%%%%%%%%%%%%%%%%%%%%%%%%%%%%%%%%%%%%%%%%%%%%%%%%%%%%%%%%%%%%%%%%%%%%%%%%%%%%%%%

To deal with Problem \ref{prob-1}, we start with the first item by formulating and solving the optimization problem, the solution of which lays a foundation for the following two items in the next section.

%%--------------------------------------------------------------------------------------------------------
\subsection{Cost Function and Optimization Problem}
\label{subsec-approbisimu}
%%--------------------------------------------------------------------------------------------------------

To evaluate the scheduling and allocation, we first consider the QoS performance. Here, we propose the following average mean response time \cite{Tang2014dynamic} as the QoS performance:
\begin{align}
\label{eqn-7}
\mathcal{T}:=\scaleobj{0.8}{\sum_{i=1}^{n}}\lambda^{-1}u_{i}(\gamma_{i}-u_{i})^{-1}.
\end{align}
Since $\mathcal{T}$ needs to be non-negative and bounded, $\gamma_{i}>u_{i}$ is required, which also shows the reasonability of Assumption \ref{asp-1}. Under Assumption \ref{asp-1}, $\mathcal{T}$ is convex with respect to $\gamma_{i}$ and $u_{i}$, respectively.

To evaluate cost of all computing nodes, we investigate the power consumption via the total service cost defined below:
\begin{align}
\label{eqn-8}
\mathcal{C}:=\scaleobj{0.8}{\sum_{i=1}^{n}}\lambda^{-1}u_{i}\varphi_{i}(\gamma_{i}),
\end{align}
where $\varphi_{i}(\gamma_{i}): \mathbb{R}^{+}\rightarrow\mathbb{R}^{+}$ is the service cost for each computing node. The service cost $\varphi_{i}(\gamma_{i})$ includes the routing cost and the computing cost. The routing cost is assumed to be a constant, whereas the computing cost is related to and non-decreasing in the service rate $\gamma_{i}$. That is, the computing cost will increase if more computing resources are requested by each computing node. The computing cost includes the cost of the power used by each computing node and the memory costs required by each computing node. In this respect, for each computing node, its service cost is defined as follows:
\begin{align}
\label{eqn-9}
\varphi_{i}(\gamma_{i}):=a_{i}\gamma^{b_{i}}_{i}+c_{i}\gamma_{i}+d_{i}, \quad \forall i\in\mathcal{N},
\end{align}
where the first item is the power cost which is monomial in the service rate; the second item is the cost of processor and storage memory; the last item is the routing cost; and $a_{i}, c_{i}, d_{i}>0$ and $b_{i}>1$. It is easy to check that $\varphi_{i}(\gamma_{i})$ is convex, and its derivative with respect to $\gamma_{i}\geq 0$, i.e.,  $\varphi^{\prime}_{i}(\gamma_{i}):=a_{i}b_{i}\gamma^{b_{i}-1}_{i}+c_{i}$, is positive and increasing with respect to $\gamma_{i}$. Since the service rate is upper bounded in practical systems, we can define $\varphi_{i}(\gamma_{i})=\infty$ if $\gamma_{i}>\gamma^{\max}_{i}$; see, e.g., \cite{Tang2014dynamic}.

Combining \eqref{eqn-7} and \eqref{eqn-8} yields the cost function $\mathcal{J}:=\mathcal{T}+K\mathcal{C}$, where $K>0$ is a fixed weight. Hence, the following optimization problem can be formulated
\begin{subequations}
\label{eqn-10}
\begin{align}
\label{eqn-10a}
\min&\quad \scaleobj{0.8}{\sum_{i=1}^{n}}\lambda^{-1}u_{i}((\gamma_{i}-u_{i})^{-1}+K\varphi_{i}(\gamma_{i})) \\
\label{eqn-10b}
\text{s.t.}& \quad \scaleobj{0.8}{\sum_{i=1}^{n}}u_{i}=\phi\lambda,  \quad  \phi>1,  \\
\label{eqn-10c}
&\quad \gamma_{i}\leq\gamma^{\max}_{i},  \quad  0\leq u_{i}<\gamma_{i},\quad \forall i\in\mathcal{N},
\end{align}
\end{subequations}
where \eqref{eqn-10b} is from \eqref{eqn-5}, and \eqref{eqn-10c} is from Assumption \ref{asp-1}.

%%--------------------------------------------------------------------------------------------------------
\subsection{Optimal Strategy under Over-Scheduling}
%%--------------------------------------------------------------------------------------------------------

The solution to \eqref{eqn-10} is derived in this subsection. Since not all computing nodes need to be activated, we first focus on how to choose the computing nodes to be activated. Based on the cost function \eqref{eqn-10a}, we introduce the following variable for each computing node:
\begin{align}
\label{eqn-11}
\theta_{i}:=\phi\min\nolimits_{0<\gamma\leq\gamma^{\max}_{i}}\{\gamma^{-1}+K\varphi_{i}(\gamma)\}, \quad \forall  i\in\mathcal{N},
\end{align}
which can be treated as the price of each computing node \cite{Tang2014dynamic}. To show this, if the dispatcher pays one euro per unit service time and $K$ euros per unit cost, then the price of each computing node is the expected total price that the dispatcher pays to this computing node. In this way, the dispatcher should choose the computing nodes with the lowest prices to ensure the minimization of the cost function. Hence, there exists a threshold for each computing node to decide whether this computing node is activated to schedule and process the requests. That is, each computing node is activated only when its price is below its threshold. For each $i\in\mathcal{N}$ and any $\theta\geq\phi Kd_{i}$, we consider the equation with respect to $\gamma$:
\begin{align}
\label{eqn-12}
K\phi\varphi_{i}(\gamma)+K\phi\gamma\varphi^{\prime}_{i}(\gamma)=\theta.
\end{align}
In \eqref{eqn-12}, $K\phi\varphi_{i}(\gamma)+K\phi\gamma\varphi^{\prime}_{i}(\gamma)$ is lower bounded by $\phi Kd_{i}$ and upper bounded due to $\gamma\leq\gamma^{\max}_{i}$. Let $\gamma=g_{i}(\theta)$ be the solution to \eqref{eqn-12}. Since $\varphi_{i}(\gamma)$ and $\varphi^{\prime}_{i}(\gamma)$ are positive and increasing, $g_{i}(\cdot)$ is an increasing function. If $\theta$ exceeds certain bound $\theta^{\max}_{i}$, which is related to $\gamma^{\max}_{i}$, then no solution to \eqref{eqn-12} exists. In this case, let $g_{i}(\theta)=g_{i}(\theta^{\max}_{i})$ for all $\theta\geq\theta^{\max}_{i}$ and $i\in\mathcal{N}$.

\begin{theorem}
\label{thm-1}
Consider the problem \eqref{eqn-10}. Let Assumption \ref{asp-1} hold and $\min_{i\in\mathcal{N}}\{\phi Kd_{i}\}<\theta_{1}\leq\ldots\leq\theta_{n}\leq\theta_{n+1}=\max_{i\in\mathcal{N}}\{g^{-1}_{i}(\gamma^{\max}_{i})\}$ with the inverse function $g^{-1}_{i}$. The optimal scheduling rates and service capacities are
\begin{align}
\label{eqn-13}
\gamma^{\opt}_{i}&=\left\{\begin{aligned}
&g_{i}(\theta), && \text{ for } 1\leq i\leq n^{\ast}, \\
&0, && \text{ for } n^{\ast}<i\leq n,
\end{aligned}\right. \\
\label{eqn-14}
u^{\opt}_{i}&=\left\{\begin{aligned}
&\gamma_{i}^{\opt}-\sqrt{\frac{\phi\gamma^{\opt}_{i}}{\theta-K\phi\varphi_{i}(\gamma^{\opt}_{i})}}, && \text { for } 1\leq i\leq n^{\ast},  \\
&0, && \text { for } n^{\ast}<i\leq n,
\end{aligned}\right.
\end{align}
where $\theta\in(\theta_{n^{\ast}}, \theta_{n^{\ast}+1}]$ is such that $\sum_{i=1}^{n}u_{i}^{\opt}=\phi\lambda$ and $n^{\ast}:=\arg\max_{k\in\mathcal{N}}\{\theta_{k}: \sum_{i=1}^{k}(g_{i}(\theta_{k})-1/\sqrt{K\varphi^{\prime}_{i}(g_{i}(\theta_{k}))})\leq\phi\lambda\}$.
\end{theorem}

\begin{proof}
The Lagrangian for \eqref{eqn-10} is defined as $\mathcal{L}:=\sum_{i=1}^{n}\lambda^{-1}u_{i}((\gamma_{i}-u_{i})^{-1}+K\varphi_{i}(\gamma_{i}))-\theta(\sum_{i=1}^{n}\phi^{-1}\lambda^{-1}u_{i}-1)
-\sum_{i=1}^{n}\mathbf{b}_{i}u_{i}-\sum_{i=1}^{n}\mathbf{h}_{i}(\gamma_{i}-u_{i})+\sum_{i=1}^{n}\mathbf{k}_{i}(\gamma_{i}-\gamma^{\max}_{i})$, where $\theta, \mathbf{b}_{i}, \mathbf{h}_{i}, \mathbf{k}_{i}\in\mathbb{R}^{+}$ are Lagrange multipliers. Since the cost function is not convex with respect to $(\gamma_{i}, u_{i})$, the problem \eqref{eqn-10} is not convex. To derive the optimal solution, we show the satisfaction of the KKT condition for the optimal solution first, and then the uniqueness of the optimal solution.

Since not all computing nodes need to be activated, the set $\mathcal{N}$ is partitioned into the inactivated part $\mathcal{I}\subset\mathcal{N}$ and the activated part $\mathcal{A}\subset\mathcal{N}$. Thus, $\mathcal{I}\cap\mathcal{A}=\varnothing$, $\mathcal{I}\cup\mathcal{A}=\mathcal{N}$, and $u_{i}=\gamma_{i}=0$ for all $i\in\mathcal{I}$. In this respect, the optimization problem \eqref{eqn-10} is rewritten equally as
\begin{subequations}
\label{eqn-15}
\begin{align}
\label{eqn-15a}
\min&\quad \scaleobj{0.8}{\sum\nolimits_{i\in\mathcal{A}}}\lambda^{-1}u_{i}((\gamma_{i}-u_{i})^{-1}+K\varphi_{i}(\gamma_{i})) \\
\label{eqn-15b}
\text{s.t.}& \quad G_{1}(u, \gamma):=\scaleobj{0.8}{\sum\nolimits_{i\in\mathcal{A}}}u_{i}-\phi\lambda=0,  \\
\label{eqn-15c}
&\quad G_{2i}(u, \gamma):=\gamma_{i}-\gamma^{\max}_{i}\leq0,  \quad \forall i\in\mathcal{A}, \\
\label{eqn-15d}
&\quad G_{3i}(u, \gamma):=u_{i}-\gamma_{i}<0,\quad \forall i\in\mathcal{A},  \\
\label{eqn-15e}
&\quad G_{4i}(u, \gamma):=-u_{i}<0,\quad \forall i\in\mathcal{A},
\end{align}
\end{subequations}
where $u:=(u_{1}, \ldots, u_{n})\in\mathbb{R}^{n}$ and $\gamma:=(\gamma_{1}, \ldots, \gamma_{n})\in\mathbb{R}^{n}$. From \eqref{eqn-15}, $\nabla G_{1}(u, \gamma)$ is only related to $u$ and linearly independent, whereas $\nabla G_{2}(u, \gamma):=(\nabla G_{21}(u, \gamma), \ldots, \nabla G_{2n}(u, \gamma))$ is only related to $\gamma$ and linearly independent. Hence, we can easily check that $\nabla G_{1}(u, \gamma)$ and $\nabla G_{2i}(u, \gamma)$ with $i\in\mathcal{A}$ are linearly independent. From \cite[Def. 4.1]{Hoheisel2009abadie} and \cite[Thm. 4.3]{Hoheisel2009abadie}, the Guignard constraint qualification (GCQ) holds at $(u, \gamma)$, and further from \cite[Thm. 6.1.4]{Bazaraa1976foundations}, the optimal solution satisfies the following KKT conditions:
\begin{subequations}
\label{eqn-16}
\begin{align}
\label{eqn-16a}
&\frac{\partial\mathcal{L}}{\partial u_{i}}=\frac{\gamma_{i}}{\lambda(\gamma_{i}-u_{i})^{2}}+\frac{\phi K\varphi_{i}(\gamma_{i})-\theta}{\phi\lambda}-\mathbf{b}_{i}+\mathbf{h}_{i}=0, \\
\label{eqn-16b}
&\frac{\partial\mathcal{L}}{\partial\gamma_{i}}=\frac{-u_{i}}{\lambda(\gamma_{i}-u_{i})^{2}}+\frac{Ku_{i}\varphi^{\prime}_{i}(\gamma_{i})}{\lambda}-\mathbf{h}_{i}+\mathbf{k}_{i}=0, \\
\label{eqn-16c}
& \scaleobj{1}{\sum_{i=1}^{n}}\phi^{-1}\lambda^{-1}u_{i}=1, \quad  \scaleobj{0.7}{\sum_{i=1}^{n}}\mathbf{b}_{i}u_{i}=0,   \\
\label{eqn-16d}
& \scaleobj{1}{\sum_{i=1}^{n}}\mathbf{h}_{i}(\gamma_{i}-u_{i})=0, \quad \scaleobj{0.7}{\sum_{i=1}^{n}}\mathbf{k}_{i}(\gamma_{i}-\gamma^{\max}_{i})=0.
\end{align}
\end{subequations}
From \eqref{eqn-16c}-\eqref{eqn-16d} and $u_{i}<\gamma_{i}$, $\mathbf{h}_{i}\equiv0$. Then, from \eqref{eqn-16a},
\begin{align}
\label{eqn-17}
u^{\opt}_{i}=\max\left\{0, \gamma^{\opt}_{i}-\sqrt{\frac{\phi\gamma^{\opt}_{i}}{\theta+\phi\lambda\mathbf{b}_{i}-K\phi\varphi_{i}(\gamma^{\opt}_{i})}}\right\},
\end{align}
where $\gamma^{\opt}_{i}$ is the optimal service rate. If $\gamma_{i}<\gamma^{\max}_{i}$, then $\mathbf{k}_{i}=0$ from \eqref{eqn-16d}, and further from \eqref{eqn-16b},
\begin{align}
\label{eqn-18}
u^{\opt}_{i}((\gamma^{\opt}_{i}-u^{\opt}_{i})^{-2}-K\varphi^{\prime}_{i}(\gamma^{\opt}_{i}))=0.
\end{align}
If $u^{\opt}_{i}>0$, then $\mathbf{b}_{i}=0$ from \eqref{eqn-16c}, and from \eqref{eqn-17}-\eqref{eqn-18},
\begin{align}
\label{eqn-19}
\gamma^{\opt}_{i}(\gamma^{\opt}_{i}-u^{\opt}_{i})^{-2}&=\phi^{-1}\theta-K\varphi_{i}(\gamma^{\opt}_{i}),
\end{align}
which shows $\gamma^{\opt}_{i}=\mathbf{g}_{i}(\theta)$ from \eqref{eqn-12}. If $u^{\opt}_{i}=0$ for some $i\in\mathcal{N}$, then the optimal value of $\gamma^{\opt}_{i}$ has no effects on the cost function in \eqref{eqn-10a}, and can be chosen arbitrarily from $[0, \gamma^{\max}_{i}]$. In this case, we can still choose $\gamma^{\opt}_{i}=\mathbf{g}_{i}(\theta)$, whereas $\mathbf{b}_{i}$ needs to be chosen appropriately to ensure the KKT conditions \eqref{eqn-16a}-\eqref{eqn-16d}. If $\gamma_{i}=\gamma^{\max}_{i}$, then $\mathbf{k}_{i}\geq0$ can be chosen arbitrarily. In this case, we can follow the above argument to derive the same optimal values. On the other hand, if $u^{\opt}_{i}>0$, then $\mathbf{b}_{i}=0$ and from \eqref{eqn-17}, $\theta>\phi/\gamma^{\opt}_{i}+K\phi\varphi_{i}(\gamma^{\opt}_{i})\geq\theta_{i}$, where $\theta_{i}$ is defined in \eqref{eqn-11}. If $u^{\opt}_{i}=0$, then from \eqref{eqn-17} and \eqref{eqn-12},
\begin{align}
\label{eqn-20}
\theta+\phi\lambda\mathbf{b}_{i}&\leq\phi\mathbf{g}^{-1}_{i}(\theta)+K\varphi_{i}(\mathbf{g}_{i}(\theta)) \nonumber  \\
& \leq\theta-\phi\mathbf{g}_{i}(\theta)(K\varphi^{\prime}_{i}(\mathbf{g}_{i}(\theta))-\mathbf{g}^{-2}_{i}(\theta)).
\end{align}
From $\mathbf{b}_{i}\geq0$, \eqref{eqn-20} holds if either $\mathbf{g}_{i}(\theta)=0$ or $K\varphi^{\prime}_{i}(\mathbf{g}_{i}(\theta))-\mathbf{g}^{-2}_{i}(\theta)\leq0$. If $K\varphi^{\prime}_{i}(\mathbf{g}_{i}(\theta))-\mathbf{g}^{-2}_{i}(\theta)\leq0$, then $\mathbf{g}_{i}(\theta)\leq\bar{\gamma}_{i}$, where $\bar{\gamma}_{i}=\mathbf{g}_{i}(\theta_{i})$ is from \cite[Lem. 1]{Ren2021optimal}. Since $\mathbf{g}_{i}(\cdot)$ is increasing, $\mathbf{g}_{i}(\theta)\leq\mathbf{g}_{i}(\theta_{i})$ implies $\theta\leq\theta_{i}$. Hence, in the derived optimal strategy, the computing nodes are chosen via the increasing order of $\theta_{i}$. In addition, $\sum_{i=1}^{n}u^{\opt}_{i}=\phi\lambda$ needs to be satisfied, thereby resulting in the constraint on $n^{\ast}$.

Next, we show the uniqueness of the derived solution via the contradiction. Let $(\bar{u}_{\opt}, \bar{\gamma}_{\opt})$ and $(\hat{u}_{\opt}, \hat{\gamma}_{\opt})$ be two different solutions to \eqref{eqn-10}. Since $\sum_{i=1}^{n}\bar{u}^{\opt}_{i}=\sum_{i=1}^{n}\hat{u}^{\opt}_{i}=\phi\lambda$, we have $\bar{u}^{\opt}_{i}\neq\hat{u}^{\opt}_{i}$ for all $i\in\{1, \ldots, \max\{\bar{n}^{\ast}, \hat{n}^{\ast}\}\}$, which further implies from \eqref{eqn-14} that $\bar{\gamma}^{\opt}_{i}\neq\hat{\gamma}^{\opt}_{i}$ for all $i\in\{1, \ldots, \max\{\bar{n}^{\ast}, \hat{n}^{\ast}\}\}$. However, For the equation $\sum_{i=1}^{n}(x_{i}-\sqrt{\frac{\phi x_{i}}{\theta-K\phi\varphi_{i}(x_{i})}})=\phi\lambda$ with $\phi>1$ and $\lambda>0$, its solution is unique from \cite{Fujisawa1971some}, and thus $\bar{\gamma}^{\opt}_{i}=\hat{\gamma}^{\opt}_{i}$ for all $i\in\{1, \ldots, n\}$, which results in a contradiction. Therefore, the solution to the problem \eqref{eqn-10} is unique, which completes the proof.
\end{proof}

Theorem \ref{thm-1} shows how to determine scheduling and service rates in the over-scheduling case, which is different from our previous work \cite{Ren2021optimal} where over-scheduling is not considered. In addition, Theorem \ref{thm-1} provides bounds on scheduling and service rates related to over-scheduling the allowable magnitude of which can be controlled by the tuning parameter $\phi>1$.

%%%%%%%%%%%%%%%%%%%%%%%%%%%%%%%%%%%%%%%%%%%%%%%%%%%%%%%%%%%%%%%%%%%%%%%%%%%%%%%%%%%%%%%%%%%%%%%%%%%%%%%%%%
\section{A-AIMD Mechanism}
\label{sec-AIMDschedule}
%%%%%%%%%%%%%%%%%%%%%%%%%%%%%%%%%%%%%%%%%%%%%%%%%%%%%%%%%%%%%%%%%%%%%%%%%%%%%%%%%%%%%%%%%%%%%%%%%%%%%%%%%%

In the derived optimal strategy, the evolution of the scheduling rates is not involved, and the effect of $\phi>1$ on the A-AIMD mechanism deserves further study. Motivated by these observations, we address the A-AIMD mechanism in this section. In particular, the convergence property of the A-AIMD mechanism is derived, which is further combined with the optimal strategy to design the AIMD parameters.

\subsection{Convergence under Over-Scheduling}
\label{subsec-knowncase}

To show convergence, we first transform the A-AIMD mechanism \eqref{eqn-6} into a switched system whose stability guarantees that scheduling rates converge to certain fixed points.

We define the following functions:
\begin{align}
\label{eqn-21}
\begin{aligned}
\mathsf{G}(u, v)&=\max\{\mathsf{G}_{1}(u, v), \mathsf{G}_{2}(v)\}, \\
\mathsf{G}_{1}(u, v)&=0.5\scaleobj{0.8}{\sum_{i=1}^{n}}(\beta_{i}u_{i}+v_{i}), \quad \mathsf{G}_{2}(v)=\phi^{-1}\scaleobj{0.8}{\sum_{i=1}^{n}}v_{i},
\end{aligned}
\end{align}
where $u = (u_1,\ldots,u_n)$, $v = (v_1,\ldots,v_n)$. Therefore, the ETM \eqref{eqn-6b} is rewritten as
\begin{align}
\label{eqn-22}
t_{k+1}&=\min\{t>t_{k}: \mathsf{G}(u(t_{k}), u(t))=\lambda\},
\end{align}
Let $\mathsf{T}_{1}(t_{k})\geq0$ ($\mathsf{T}_{2}(t_{k})\geq0$) be the inter-event period when $\mathsf{G}_{1}=\lambda$ ($\mathsf{G}_{2}=\lambda$). If $\mathsf{G}(u(t_{k}), u(t_{k+1}))=\mathsf{G}_{1}(u(t_{k}), u(t_{k+1}))$, then $\mathsf{T}_{1}(t_{k})\leq\mathsf{T}_{2}(t_{k})$ and $\mathsf{T}_{k}=\mathsf{T}_{1}(t_{k})$, where $\mathsf{T}_{k}$ is defined in \eqref{eqn-3}. In this case, combining $\mathsf{G}_{1}(u(t_{k}), u(t_{k+1}))=\lambda$ and \eqref{eqn-4} yields
\begin{align}
\label{eqn-23}
\begin{aligned}
\mathsf{T}_{1}(t_{k})&=\frac{2(\lambda-\sum_{i=1}^{n}\beta_{i}u_{i}(t_{k}))}{\sum_{i=1}^{n}\alpha_{i}}, \\
u_i(t_{k+1})&=\beta_{i}u_i(t_{k})+\alpha_{i}\mathsf{T}_{1}(t_{k}).
\end{aligned}
\end{align}
If $\mathsf{G}(u(t_{k}), u(t_{k+1}))=\mathsf{G}_{2}(u(t_{k+1}))$, then $\mathsf{T}_{k}=\mathsf{T}_{2}(t_{k})<\mathsf{T}_{1}(t_{k})$, and $\mathsf{T}_{2}(t_{k})$ is defined via $\mathsf{G}_{2}(u(t_{k+1}))=\lambda$. That is,
\begin{align}
\label{eqn-24}
\begin{aligned}
\mathsf{T}_{2}(t_{k})&=\frac{\phi\lambda-\sum_{i=1}^{n}\beta_{i}u_{i}(t_{k})}{\sum_{i=1}^{n}\alpha_{i}}, \\
u_i(t_{k+1})&=\beta_{i}u_i(t_{k})+\alpha_{i}\mathsf{T}_{2}(t_{k}).
\end{aligned}
\end{align}

In view of \eqref{eqn-23}-\eqref{eqn-24}, the aggregate dynamics of the A-AIMD mechanism \eqref{eqn-6} can be expressed by the following state-dependent switched system.
\begin{align}
\label{eqn-25}
u(t_{k+1})=\left\{\begin{aligned}
&A_{1} u(t_{k})+B_{1} & \text {if }& \beta^{\top}u(t_{k})\geq(2-\phi)\lambda \\
&A_{2} u(t_{k})+B_{2} & \text {if }& \beta^{\top}u(t_{k})<(2-\phi)\lambda
\end{aligned}\right.
\end{align}
with $u(t_{k})=(u_{1}(t_{k}), \ldots, u_{n}(t_{k}))\in\mathbb{R}^{n}, A_{1}=\diag\{\beta\}-\frac{2\alpha\beta^{\top}}{\sum_{i=1}^{n}\alpha_{i}}, A_{2}=\diag\{\beta\}-\frac{\alpha\beta^{\top}}{\sum_{i=1}^{n}\alpha_{i}}, B_{1}=\frac{2\lambda\alpha}{\sum_{i=1}^{n}\alpha_{i}}$ and $B_{2}=\frac{\phi\lambda\alpha}{\sum_{i=1}^{n}\alpha_{i}}$, where $\beta=(\beta_{1}, \ldots, \beta_{n})\in\mathbb{R}^{n}, \alpha=(\alpha_{1}, \ldots, \alpha_{n})\in\mathbb{R}^{n}$ and $\diag(\cdot)$ is the diagonal operator. Clearly, \eqref{eqn-25} is a discrete-time state-dependent switched system with two modes where the ETM \eqref{eqn-22} is embedded into the switching conditions. Since $\beta^{\top}u(t_{k})\geq0$, then $\phi<2$ is needed to ensure that the condition $\beta^{\top}u(t_{k})<(2-\phi)\lambda$ in \eqref{eqn-25} is reasonable. Therefore, we have $\phi\in(1, 2)$. Before showing the convergence of the scheduling rates via \eqref{eqn-25}, we first present a proposition to show the stability properties of matrices $A_1$ and $A_2$ as defined in \eqref{eqn-25}.

\begin{proposition}
\label{prop-1}
Let matrices $A_1$, $A_2$ be defined as in \eqref{eqn-25} and the matrix set be $\Sigma=\{A_1, A_2\}$. The joint spectral radius of $\Sigma$, namely, $\mathsf{JSR}(\Sigma)<1$ for all $\alpha_i > 0, \beta_i\in (0, 1)$, $i\in\mathcal{N}$.
\end{proposition}

\begin{proof}
Define the matrix $C:=\diag\{\bar{\alpha}\}(\diag\{\bar{\beta}\})^{-1}$, where $\bar{\alpha}:=(\sqrt{\alpha_{1}/(\sum_{i=1}^{n}\alpha_{i})}, \ldots, \sqrt{\alpha_{n}/(\sum_{i=1}^{n}\alpha_{i})})\in\mathbb{R}^{n}$ and $\bar{\beta}:=(\sqrt{\beta_{1}}, \ldots, \sqrt{\beta_{n}})\in\mathbb{R}^{n}$. Following the similar mechanism as in \cite[Sec. IV]{Vlahakis2021AIMD}, it is easy to see that matrices $\bar{A}_{1}=C^{-1}A_{1}C=\diag\{\beta\}-2\zeta^{\top}\zeta$ and $\bar{A}_{2}=C^{-1}A_{2}C=\diag\{\beta\}-\zeta^{\top}\zeta$ are symmetric, where $\zeta=(\bar{\alpha}_{1}\bar{\beta}_{1}, \ldots, \bar{\alpha}_{n}\bar{\beta}_{n})\in\mathbb{R}^{n}$. From \cite[Thm. 2]{Vlahakis2021AIMD}, the absolute values of all eigenvalues of $A_{1}$ are strictly less than $\beta_{\max}<1$ where $\beta_{\max}:=\max\{\beta_{i}: i\in\mathcal{N}\}$, which implies that $A_{1}$ is a Schur matrix. Based on the similarity property of matrices, $A_{2}$ and $\bar{A}_{2}$ have the same spectrum, which is denoted as $\mathcal{S}(A_{2}):=\{\xi_{1}, \ldots, \xi_{n}\}$ with an increasing order. Let $\{\beta_{1}, \ldots, \beta_{n}\}$ be in an increasing order, and from \cite[Thm. 1]{Vlahakis2021AIMD}, $\xi_{1}\leq\beta_{1}\leq\xi_{2}\leq\beta_{2}\leq\ldots\leq\xi_{n}\leq\beta_{n}$, which implies $\xi_{i}\in(0, \beta_{\max})$ for all $i\in\{2, \ldots, n\}$. Note that $A_{2}$ can be rewritten as $A_{2}=(I-\bar{\alpha}\mathbf{1}^{\top})\diag\{\beta\}$, where $\mathbf{1}\in\mathbb{R}^{n}$ is the vector whose components are 1. Here, $\mathcal{S}(\bar{\alpha}\mathbf{1}^{\top})=\{\mathbf{1}^{\top}\bar{\alpha}, 0, \ldots, 0\}$ with $\mathbf{1}^{\top}\bar{\alpha}=1$, and $\mathcal{S}(I-\bar{\alpha}\mathbf{1}^{\top})=\{0, 1, \ldots, 1\}$. The determinant of $A_{2}$ is $\dett(A_{2})=\prod^{n}_{i=1}\xi_{i}=\dett(I-\bar{\alpha}\mathbf{1}^{\top})\dett(\diag\{\beta\})=0$. Since $\xi_{i}\neq 0$ for all $i\in\{2, \ldots, n\}$, we have $\xi_{1}=0$. Therefore, $\xi_{i}\in[0, \beta_{\max})$ for all $i\in\mathcal{N}$, and $A_{2}$ is a Schur matrix. Let $\bar{\Sigma} = C^{-1}\Sigma C = \{\bar{A}_{1},\;\bar{A}_{2}\}$. Since, matrices $\bar{A}_1$, $\bar{A}_2$ are symmetric, from \cite[Cor. 2.3.]{Jungers2009joint}, we have that $\mathsf{JSR}(\bar{\Sigma})=\max\{\rho(A): A\in\bar{\Sigma}\}<\beta_{\max}<1$, where $\rho(A)$ is the spectral radius of $A$. Finally, from \cite[Prop. 1.3.]{Jungers2009joint}, we have $\mathsf{JSR}(\Sigma)=\mathsf{JSR}(\bar{\Sigma})<\beta_{\max}<1$.
\end{proof}

Roughly, Proposition \ref{prop-1} shows that the linear part of \eqref{eqn-25} is stable under arbitrary switching. In the following theorem, we show that under an appropriate selection of the AIMD parameters $\alpha_{i}, \beta_{i}$, the system \eqref{eqn-25} converges to a unique fixed point while the backlog $\delta(t_k)$ in the dispatcher is bounded for all $t_k\geq0$.

\begin{theorem}
\label{thm-2}
Consider the A-AIMD mechanism \eqref{eqn-25} with $\phi\in(1, 2)$ and $\alpha_{i}>0, \beta_{i}\in(0, 1)$ satisfying
\begin{equation}
\label{eqn-26}
1+\frac{\sum_{i=1}^{n}\alpha_{i}}{\sum_{i=1}^{n}\frac{1+\beta_{i}}{1-\beta_{i}}\alpha_{i}}=\phi.
\end{equation}
Then the following holds.
\begin{enumerate}
  \item The system \eqref{eqn-25} converges to a fixed point $u_{\con}:=(u^{\con}_{1}, \ldots, u^{\con}_{n})\in\mathbb{R}^{n}$ with
\begin{equation}
\label{eqn-27}
u^{\con}_{i}:=(1-\beta_{i})^{-1}\alpha_{i}\mathsf{T}^{\ast}, \quad  \mathsf{T}^{\ast}=\frac{2\lambda}{\sum_{i=1}^{n}\frac{1+\beta_{i}}{1-\beta_{i}}\alpha_{i}}.
\end{equation}

  \item The dispatcher's backlog $\delta$ in \eqref{eqn-2} is bounded.
\end{enumerate}
\end{theorem}

\begin{proof}
To begin with, we show the first item. Since the matrices $A_1$ and $A_2$ in \eqref{eqn-25} are Schur, there are unique fixed points $\hat{u}_{\con}:=(\hat{u}^{\con}_{1}, \ldots, \hat{u}^{\con}_{n})$ and $\tilde{u}_{\con}:=(\tilde{u}^{\con}_{1}, \ldots, \tilde{u}^{\con}_{n})$ such that $\hat{u}_{\con}=A_{1}\hat{u}_{\con}+B_{1}$ and $\tilde{u}_{\con} = A_{2}\tilde{u}_{\con}+B_{2}$. From \cite{Ren2021optimal, Vlahakis2021AIMD}, we have
\begin{align}
\label{eqn-28}
\hat{u}^{\con}_{i}:=(1-\beta_{i})^{-1}\alpha_{i}\mathsf{T}^{\ast}_{1}, \quad
\tilde{u}^{\con}_{i}:=(1-\beta_{i})^{-1}\alpha_{i}\mathsf{T}^{\ast}_{2},
\end{align}
where $\mathsf{T}^{\ast}_{1}, \mathsf{T}^{\ast}_{2}>0$ are the inter-event periods associated to the fixed points in Modes 1 and 2, respectively. From \eqref{eqn-23}-\eqref{eqn-24} and \eqref{eqn-26}, we have
\begin{equation}
\label{eqn-29}
2\left(\lambda-\scaleobj{0.8}{\sum_{i=1}^{n}}\beta_{i}\hat{u}^{\con}_{i}\right)=\phi\lambda-\scaleobj{0.8}{\sum_{i=1}^{n}}\beta_{i}\tilde{u}^{\con}_{i}.
\end{equation}
Combining \eqref{eqn-29} with \eqref{eqn-28} yields $\mathsf{T}^{\ast}_{1}=\mathsf{T}^{\ast}_{2}=\mathsf{T}^{\ast}$ with $\mathsf{T}^{\ast}$ defined in \eqref{eqn-27}. Hence, $\hat{u}_{\con}=\tilde{u}_{\con}=u_{\con}$ with $u_{\con}:=(u^{\con}_{1}, \ldots, u^{\con}_{n})$. Next, we show that the system \eqref{eqn-25} converges to the fixed point as in \eqref{eqn-27}. We define the error $e(t_{k}):=u(t_{k})-u_{\con}$, whose dynamics is given by
\begin{align}
\label{eqn-30}
e(t_{k+1})=\left\{\begin{aligned}
&A_{1}e(t_{k}) & \text {if }& \beta^{\top}e(t_{k})\geq0 \\
&A_{2}e(t_{k}) & \text {if }& \beta^{\top}e(t_{k})<0.
\end{aligned}\right.
\end{align}
From Proposition \ref{prop-1} and \cite[Cor. 1.1]{Jungers2009joint}, the system \eqref{eqn-30} is asymptotically stable, which means that the origin is a fixed point for \eqref{eqn-30}, i.e., $\lim_{k\to\infty}e(t_k)=0$, or $\lim_{k\to\infty}u(t_k)=u_{\con}$, with $u_i^{\con}$ given in \eqref{eqn-27}, which completes the proof of the first item.

Next, we show the second item. From \eqref{eqn-2}, the backlog $\delta(t_{k+1})$ is given below:
\begin{align*}
\delta(t_{k+1})&=\delta(t_{k})+\int^{t_{k+1}}_{t_{k}}(\lambda-u(t))dt \\
&=\delta(t_{k})+\lambda\mathsf{T}_{k}-0.5(\beta^{\top}u(t_{k})+\mathbf{1}^{\top}u(t_{k+1}))\mathsf{T}_{k}.
\end{align*}
In Mode 1, we have from the event-triggering condition that $0.5(\beta^{\top}u(t_{k})+\mathbf{1}^{\top}u(t_{k+1}))=\lambda$, which implies that $\delta(t_{k+1})=\delta(t_{k})$.
In Mode 2, the event-triggering condition results in $\mathbf{1}^{\top}u(t_{k+1})=\phi\lambda$, and thus
\begin{align*}
\delta(t_{k+1})&=\delta(t_{k})+\frac{(2-\phi)\lambda-\beta^{\top}u(t_{k})}{2}\mathsf{T}_{k},
\end{align*}
which, from \eqref{eqn-29}, becomes
\begin{align*}
\delta(t_{k+1})&=\delta(t_{k})+\frac{\beta^{\top}u_{\con}-\beta^{\top}u(t_{k})}{2}\mathsf{T}_{k}\\
&=\delta(t_{k})-0.5\beta^{\top}e(t_{k})\mathsf{T}_{k}.
\end{align*}
Therefore, the dynamics of $\delta(t_{k+1})$ is
\begin{align}
\label{eqn-31}
\delta(t_{k+1})=\left\{\begin{aligned}
&\delta(t_{k}) & \text {if }& \beta^{\top}e(t_{k})\geq0 \\
&\delta(t_{k})-0.5\beta^{\top}e(t_{k})\mathsf{T}_{k} & \text {if }& \beta^{\top}e(t_{k})<0.
\end{aligned}\right.
\end{align}

We note that $\mathsf{T}_{k}=t_{k+1}-t_{k}$ and that $t_{k+1}$ is determined via \eqref{eqn-6b}. Hence, for any given $phi\in(1, 2), \alpha_{i}>0, \beta_{i}\in(0, 1)$, we have $\mathsf{T}_{k}$ is bounded to ensure $\sum_{i=1}^{n}(\beta_{i}u_{i}(t_{k})+0.5\alpha_{i}\mathsf{T}_{k})=\lambda$ or $\sum_{i=1}^{n}(\beta_{i}u_{i}(t_{k})+\alpha_{i}\mathsf{T}_{k})=\phi\lambda$. The boundedness of $\mathsf{T}_{k}$ can be obtained by considering the extreme case where $\beta_{i}u_{i}(t_{k})\equiv0$ for all $i\in\mathcal{N}$. In this case, $\sum_{i=1}^{n}0.5\alpha_{i}\mathsf{T}_{k}=\lambda$ implies $\mathsf{T}_{k}\leq2\lambda/(\sum_{i=1}^{n}\alpha_{i})$ and $\sum_{i=1}^{n}\alpha_{i}\mathsf{T}_{k}=\phi\lambda$ implies $\mathsf{T}_{k}\leq\phi\lambda/(\sum_{i=1}^{n}\alpha_{i})$. Therefore, the upper bound of $\mathsf{T}_{k}$ is finite and assumed to be $\mathsf{T}_{\max}>0$. More precisely, $\mathsf{T}_{\max}=\min\{\tau\in\mathbb{R}^{+}: t_{k+1}-t_{k}\leq\tau, k\geq0\}$, where $t_{k}, t_{k+1}$ are defined in \eqref{eqn-6b}. From \eqref{eqn-31}, we have the following bound inequality
\begin{align}
\label{eqn-32}
\delta(t_{k+1})\leq\delta(t_{k})+0.5|\beta^{\top}e(t_{k})|\mathsf{T}_{\max}.
\end{align}
By implementing \eqref{eqn-32} iteratively from $t_{0}$ to $t_{k+1}$, we obtain
\begin{align}
\label{eqn-33}
\delta(t_{k+1})&\leq\delta(t_{0})+0.5\sum^{k}_{l=0}|\beta^{\top}e(t_{l})|\mathsf{T}_{\max} \nonumber \\
&\leq\delta(t_{0})+0.5\beta_{\max}\sum^{k}_{l=0}\|e(t_{l})\|\mathsf{T}_{\max}.
\end{align}
In addition, the asymptotic stability of the system \eqref{eqn-25} implies that there exists $L\geq1$ such that
\begin{align}
\label{eqn-34}
\|e(t_{k})\|\leq L(\mathsf{JSR}(\Sigma))^{k}\|e(t_{0})\|\leq L\beta^{k}_{\max}\|e(t_{0})\|,
\end{align}
where the second inequality holds from Proposition \ref{prop-1}. Combining \eqref{eqn-33} and \eqref{eqn-34} yields
\begin{align*}
\delta(t_{k+1})&\leq\delta(t_{0})+0.5\beta_{\max}\mathsf{T}_{\max}\sum^{k}_{l=0}L\beta^{l}_{\max}\|e(t_{0})\| \\
&\leq\delta(t_{0})+0.5\|e(t_{0})\|\mathsf{T}_{\max}L\sum^{k}_{l=0}\beta^{l+1}_{\max} \\
&\leq\delta(t_{0})+\frac{L\beta_{\max}\mathsf{T}_{\max}\|e(t_{0})\|}{2(1-\beta_{\max})}.
\end{align*}
Since both $\delta(t_{0})$ and $\|e(t_{0})\|$ are bounded at the initial time, we have the boundedness of $\delta$, which hence completes the proof of the second item.
\end{proof}

In Theorem \ref{thm-2}, the convergence property is preserved for the A-AIMD mechanism under \eqref{eqn-26}, which can be treated as the constraints on the AIMD parameters $\alpha_{i}, \beta_{i}$ and will be discussed in the next subsection.

\subsection{Tuning of AIMD Parameters}
\label{subsec-parameterdesign}

From Theorem \ref{thm-2}, we have that if the AIMD parameters take values in the set induced by \eqref{eqn-26} the resulting system \eqref{eqn-25} is stable with a unique fixed point. In this section, we combine the solutions in Theorems \ref{thm-1}-\ref{thm-2} to ensure the convergence of all scheduling rates to the optimal scheduling rates, thereby combining the optimal strategy with the A-AIMD mechanism.

From Theorem \ref{thm-1}, $u^{\opt}_{i}=\gamma^{\opt}_{i}=0$ for all $i\in\mathcal{I}$, and thus we set $\alpha_{i}=\beta_{i}=0$ directly for all $i\in\mathcal{I}$. Next, we investigate the case of $i\in\mathcal{A}$. For all $i\in\mathcal{A}$, there are two constraints on the scheduling rates: $u^{\con}_{i}=u^{\opt}_{i}$ from our goal and $u^{\opt}_{i}<\gamma^{\opt}_{i}$ from Assumption \ref{asp-1}. Hence, there exist two cases. The first case is that $(1-\beta_{i})^{-1}\alpha_{i}\mathsf{T}^{\ast}\geq\gamma^{\opt}_{i}$. For this case, from Assumption \ref{asp-1} and $u^{\opt}_{i}=u^{\con}_{i}$, we have $u^{\opt}_{i}=\gamma^{\opt}_{i}-\epsilon$, where $\epsilon>0$ is sufficiently small to avoid $u^{\opt}_{i}=\gamma^{\opt}_{i}$. From \eqref{eqn-14}, $K\varphi^{\prime}_{i}(\gamma_{i}^{\opt})=\epsilon^{-2}$. Since $\epsilon>0$ is arbitrarily small, we have from \eqref{eqn-9} and the derivative of $\varphi_{i}$ that $\gamma^{\opt}_{i}$ needs to be sufficiently large, which contradicts with the boundedness of $\gamma_{i}$. Therefore, the first case is unreasonable, and we only need to consider the second case: $(1-\beta_{i})^{-1}\alpha_{i}\mathsf{T}^{\ast}<\gamma^{\opt}_{i}$. In the second case, from \eqref{eqn-28}, $u^{\opt}_{i}=(1-\beta_{i})^{-1}\alpha_{i}\mathsf{T}^{\ast}$, which is rewritten as
\begin{align}
\label{eqn-35}
(1-\beta_{i})u^{\opt}_{i}=\alpha_{i}\mathsf{T}^{\ast}, \quad \forall i\in\mathcal{A}.
\end{align}
Due to \eqref{eqn-26} and $\mathsf{T}^{\ast}$ in \eqref{eqn-35}, the parameters $\alpha_{i}$ and $\beta_{i}$ affect each other. We next propose two methods to design the AIMD parameters.

\subsubsection{Linear polynomial based design}
Let $x_{i}:=\alpha_{i}\mathsf{T}^{\ast}, y_{i}:=\beta_{i}$. Hence, \eqref{eqn-35} equals to $x_{i}+u^{\opt}_{i}y_{i}=u^{\opt}_{i}$. Since $\phi>1$ is known \emph{a priori}, we have from \eqref{eqn-3} and \eqref{eqn-5},
\begin{align}
\label{eqn-36}
\scaleobj{0.8}{\sum_{i=1}^{n}}x_{i}=2(\phi-1)\lambda, \quad  \scaleobj{0.8}{\sum_{i=1}^{n}}(0.5x_{i}+u^{\opt}_{i}y_{i})=\lambda,
\end{align}
where, the former one is equivalent to the constraint in \eqref{eqn-26}-\eqref{eqn-27} on $x_{i}$. All these constraints can be written into a unified equation form $Az=B$ with $z:=(x, y, \phi), x:=(x_{1}, \ldots, x_{n}), y:=(y_{1}, \ldots, y_{n})$ and
\begin{align}
\label{eqn-37}
A:=\begin{bmatrix}
I_{n\times n} & \diag(u^{\top}_{\opt}) & 0 \\ \mathbf{1}^{\top} & \mathbf{0}^{\top} &  -2\lambda \\  0.5\mathbf{1}^{\top} & u^{\top}_{\opt} & 0
\end{bmatrix}, \quad
B:=\begin{bmatrix}
u_{\opt} \\ -2\lambda \\ \lambda
\end{bmatrix},
\end{align}
where $u_{\opt}:=(u^{\opt}_{1}, \ldots, u^{\opt}_{n})\in\mathbb{R}^{n}$ and $\mathbf{0}\in\mathbb{R}^{n}$ is the vector whose components are 0. Since $\phi>0$ is known \emph{a priori}, $Az=B$ is an equation with a known variable. From \cite{Strang2005linear}, we can check that $A\in\mathbb{R}^{(n+2)\times(2n+1)}$ is full-rank and its null space is non-trivial. Hence, the solution to $Az=B$ is existent and non-unique. Different approaches, e.g., the method of least squares, can be applied to find feasible solutions.

\subsubsection{Optimization-based design}
Since the above method cannot show the uniqueness of the solution, another design method is based on the optimization theory. Here the goal is to minimize the Euclidean norm of the AIMD parameters under \eqref{eqn-35}-\eqref{eqn-36}. The optimization problem is formulated below:
\begin{align}
\label{eqn-38}
\begin{aligned}
\min_{x_{i}, y_{i}} &\ \scaleobj{0.8}{\sum_{i=1}^{n}}(x^{2}_{i}+y^{2}_{i})+\phi^{2} \quad   \\
\text{ s.t. }& \ \eqref{eqn-35}-\eqref{eqn-36}, \  u^{\top}_{\opt}y<\lambda, \ x_{i}>0,\  y_{i}\in(0, 1).
\end{aligned}
\end{align}
From \eqref{eqn-35}, the cost function in \eqref{eqn-38} can be rewritten as a function of $y$ and $\phi$. Similar to the first method, $\phi$ can be involved in the cost function such that we have the co-design of the optimal over-scheduling and AIMD parameters. If $\phi$ is fixed, then the cost function is only related to $y$. Since the problem \eqref{eqn-38} is convex, the solution is unique, which hence shows the advantages of the optimization-based method.

%%%%%%%%%%%%%%%%%%%%%%%%%%%%%%%%%%%%%%%%%%%%%%%%%%%%%%%%%%%%%%%%%%%%%%%%%%%%%%%%%%%%%%%%%%%%%%%%%%%%%%%%%%
\section{Numerical Results}
\label{sec-example}
%%%%%%%%%%%%%%%%%%%%%%%%%%%%%%%%%%%%%%%%%%%%%%%%%%%%%%%%%%%%%%%%%%%%%%%%%%%%%%%%%%%%%%%%%%%%%%%%%%%%%%%%%%

\begin{figure}[!t]
\begin{center}
\begin{picture}(60, 85)
\put(-60, -15){\resizebox{60mm}{35mm}{\includegraphics[width=2.5in]{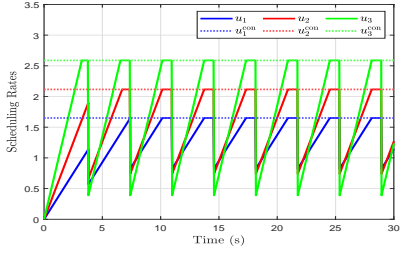}}}
\end{picture}
\end{center}
\caption{The evolution of all scheduling rates $u_{i}$, $i =\{1, 2, 3\}$.}
\label{fig-4}
\end{figure}

Consider a computing system with 3 heterogeneous computing nodes with different service capacities and power consumptions. To be specific, let the coefficients in \eqref{eqn-10} be $a_{1}=0.1, a_{2}=0.2, a_{3}=0.4, b_{1}=b_{2}=b_{3}=2, c_{1}=0.3, c_{2}=0.6, c_{3}=0.9, d_{1}=1, d_{2}=2, d_{3}=3$. In addition, $\gamma^{\max}_{1}=6, \gamma^{\max}_{2}=9, \gamma^{\max}_{3}=11$ and the weight $K=1.2$. Therefore, from \eqref{eqn-11}, $\theta_{1}=2.6373\phi, \theta_{2}=4.3524\phi, \theta_{3}=6.0133\phi$. Hence, the lightweight computing node (i.e., Node 1) is always activated. Moreover, the upper bound of $\theta_{i}$ is $\theta_{4}=201.6\phi$. With different $\phi\in[1, 2]$, the number of the activated computing nodes is presented in Fig. \ref{fig-2}. Let $\theta=13.648$
in \eqref{eqn-14}-\eqref{eqn-15}, $\phi=1.4$ and $\lambda=5.5$, and from Theorem \ref{thm-1}, the optimal rates are $u^{\opt}_{1}=4.1619, u^{\opt}_{2}=2.4024, u^{\opt}_{3}=1.4182, \gamma^{\opt}_{1}=4.9647, \gamma^{\opt}_{2}=3.0770, \gamma^{\opt}_{3}=1.9660$. The A-AIMD mechanism and over-scheduling are studied below.

\begin{figure}[!t]
\begin{center}
\begin{picture}(60, 85)
\put(-60, -15){\resizebox{60mm}{35mm}{\includegraphics[width=2.5in]{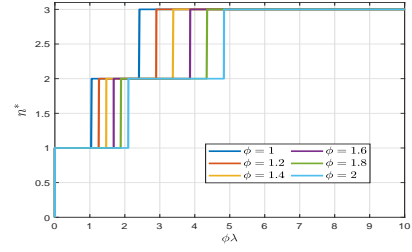}}}
\end{picture}
\end{center}
\caption{The evolution of the number $n^{\ast}$ with different $\phi\in[1, 2]$.}
\label{fig-2}
\end{figure}

First, we assume that the AIMD parameters are known. Let $\alpha_{1}=0.3, \alpha_{2}=0.5, \alpha_{3}=0.8, \beta_{1}=0.5, \beta_{2}=0.35, \beta_{3}=0.15$. From Theorem \ref{thm-2}, $\mathsf{T}^{\ast}=2.75, u^{\con}_{1}=1.65, u^{\con}_{2}=2.1154, u^{\con}_{3}=2.5882$; see Fig. \ref{fig-4}. It is easy to check that $\lambda<u^{\con}_{1}+u^{\con}_{2}+u^{\con}_{3}\leq\phi\lambda$. Next, we consider the known over-scheduling case. From Section \ref{subsec-parameterdesign}, the design of the AIMD parameters is formulated into an optimization problem \eqref{eqn-38}, and thus we have $\alpha^{\ast}_{1}\mathsf{T}^{\ast}=1.7662, \alpha^{\ast}_{2}\mathsf{T}^{\ast}=1.5559, \alpha^{\ast}_{3}\mathsf{T}^{\ast}=1.0780, \beta^{\ast}_{1}=0.5584, \beta^{\ast}_{2}=0.3197, \beta^{\ast}_{3}=0.1867$. That is, for each $i\in\{1, 2, 3\}$, $\beta_{i}$ is fixed, whereas $\alpha^{\ast}_{i}$ and $\mathsf{T}^{\ast}$ can be determined via the user requirement. The feasible choice is presented in Fig. \ref{fig-3}. For instance, a feasible choice is $\alpha^{\ast}_{1}=0.4750, \alpha^{\ast}_{2}=0.4185, \alpha^{\ast}_{3}=0.2899$ and $\mathsf{T}^{\ast}=3.7180$. Finally, if $\phi$ is embedded into \eqref{eqn-38}, then we derive the optimal coefficient $\phi^{\ast}=1.4514$ and the optimal AIMD parameters $\alpha^{\ast}_{1}\mathsf{T}^{\ast}=2.0799, \alpha^{\ast}_{2}\mathsf{T}^{\ast}=1.7086, \alpha^{\ast}_{3}\mathsf{T}^{\ast}=1.1765, \beta^{\ast}_{1}=0.5003, \beta^{\ast}_{2}=0.2888, \beta^{\ast}_{3}=0.1705$, which thus guarantee the optimal strategy in Theorem \ref{thm-1}. In particular, if $\mathsf{T}^{\ast}=1.5$, then $\alpha^{\ast}_{1}=1.3866, \alpha^{\ast}_{2}=1.1391, \alpha^{\ast}_{3}=0.7843$; see the bold points in Fig. \ref{fig-3}.

%%%%%%%%%%%%%%%%%%%%%%%%%%%%%%%%%%%%%%%%%%%%%%%%%%%%%%%%%%%%%%%%%%%%%%%%%%%%%%%%%%%%%%%%%%%%%%%%%%%%%%%%%%%%%%%%%
\section{Conclusion}
\label{sec-conclusion}
%%%%%%%%%%%%%%%%%%%%%%%%%%%%%%%%%%%%%%%%%%%%%%%%%%%%%%%%%%%%%%%%%%%%%%%%%%%%%%%%%%%%%%%%%%%%%%%%%%%%%%%%%%%%%%%%%

In this paper, we studied the resource allocation problem of computing systems under over-scheduling phenomena. We proposed an optimisation problem and derived an optimal scheduling and allocation strategy. To realize the optimal scheduling solution under over-scheduling, we proposed a feedback mechanism based on a new AIMD-like algorithm whose stability was proved by means of the boundedness of the joint spectral radius associated with an equivalent switched system. Finally, two methods for tuning AIMD parameters were proposed to fit an optimal AIMD-based mechanism to the optimal solution.

\begin{figure}[!t]
\begin{center}
\begin{picture}(65, 85)
\put(-65, -15){\resizebox{65mm}{35mm}{\includegraphics[width=2.5in]{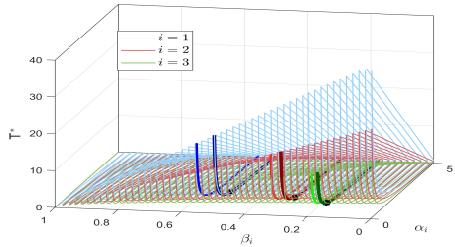}}}
\end{picture}
\end{center}
\caption{Illustration of the relation among $\alpha_{i}, \mathsf{T}^{\ast}$ and $\beta_{i}$. The curves with light colors are from \eqref{eqn-26}. The curves with blue, red and green colors are a feasible solution for the case where $\phi=1.4$. The curves with dark colors are for the optimal case $\phi^{\ast}=1.4514$.}
\label{fig-3}
\end{figure}


\end{document}